\documentclass[12pt]{amsart}
\usepackage{amscd}
\usepackage{amsmath}
\usepackage{verbatim}
\usepackage{longtable}
\usepackage{stmaryrd}
\usepackage{mathtools}
\usepackage[all, cmtip,arrow]{xy}
\usepackage{amssymb}

\numberwithin{equation}{section}

\newtheorem{thm}[equation]{Theorem}
\newtheorem{prop}[equation]{Proposition}

\newtheorem{cor}[equation]{Corollary}

\theoremstyle{definition}
\newtheorem{rem}[equation]{Remark}
\newtheorem{example}[equation]{Example}
\newtheorem{dfn}[equation]{Definition}
\newtheorem{ntt}[equation]{}

\newcommand{\llarrow}{\mathrel{\vcenter{\vbox{\offinterlineskip
\hbox{$\leftarrow$}\hbox{$\leftarrow$}}}}}
\newcommand{\lllarrow}{\mathrel{\vcenter{\vbox{\offinterlineskip
\hbox{$\leftarrow$}\hbox{$\leftarrow$}\hbox{$\leftarrow$}}}}}

\newcommand{\zz}{\mathbb{Z}}
\newcommand{\aff}{\mathrm{aff}}
\newcommand{\qq}{\mathbb{Q}} 
\newcommand{\laz}{\mathbb{L}}
\newcommand{\ff}{\mathbb{F}}
\newcommand{\op}{\mathrm{op}}

\newcommand{\ta}{\mathbb{T}}
\newcommand{\res}{\mathrm{res}}
\DeclareMathOperator{\codim}{\mathrm{codim}}
\newcommand{\Mot}{\mathcal{M}}
\DeclareMathOperator{\End}{\mathrm{End}}

\newcommand{\A}{\mathrm{A}}
\newcommand{\B}{\mathrm{B}}
\newcommand{\C}{\mathrm{C}}
\newcommand{\D}{\mathrm{D}}
\newcommand{\E}{\mathrm{E}}
\newcommand{\F}{\mathrm{F}}
\newcommand{\G}{\mathrm{G}}

\DeclareMathOperator{\Aut}{\mathrm{Aut}}
\DeclareMathOperator{\Hom}{\mathrm{Hom}}
\DeclareMathOperator{\PGL}{\mathrm{PGL}}
\DeclareMathOperator{\SL}{\mathrm{SL}}
\DeclareMathOperator{\Spec}{\mathrm{Spec}}
\DeclareMathOperator{\CH}{\mathrm{CH}}
\DeclareMathOperator{\Pic}{\mathrm{Pic}} 
\DeclareMathOperator{\ind}{\mathrm{ind}} 

\newcommand{\MS}{\mathrm{MS}}
\newcommand{\Nrd}{\mathrm{Nrd}}
\newcommand{\SB}{\mathrm{SB}}
\newcommand{\Br}{\mathop{\mathrm{Br}}}
\newcommand{\Spin}{\operatorname{\mathrm{Spin}}}
\newcommand{\GL}{\operatorname{\mathrm{GL}}}

\newcommand{\Fsep}{F_{{\mathrm{sep}}}}
\newcommand{\sep}{\mathrm{sep}}
\newcommand{\FGL}{\mathrm{FGL}}
\newcommand{\Ker}{\operatorname{Ker}}
\newcommand{\et}{{\text{et}}}

\title[Invariants and cohomology]{Cohomological invariants of algebraic groups and the Morava $K$-theory}
\author{Nikita Semenov}
\thanks{The author gratefully acknowledges the support of the Sonderforschungsbereich/Transregio 45 ``Periods,
moduli spaces, and arithmetic of algebraic varieties'' (Bonn-Essen-Mainz) and Universit\'e Paris~13.}


\date{}

\begin{document}

\begin{abstract}
In the present article we discuss different approaches to cohomological invariants of algebraic groups over a
field. We focus on the Tits algebras and on the Rost invariant
and relate them to the Morava $K$-theory.
Furthermore, we discuss oriented cohomology theories of affine varieties and the Rost motives for different
cohomology theories in the sense of Levine--Morel.
\end{abstract}

\maketitle

\section{Introduction}
The notion of a Tits algebra was introduced by Jacques Tits in his celebrated
paper on irreducible representations \cite{Ti71}. This invariant of a linear algebraic
group $G$ plays a crucial role in the computation of the $K$-theory of twisted 
flag varieties by Panin \cite{Pa94} and in the index reduction formulas by Merkurjev,
Panin and Wadsworth \cite{MPW96}. It has important applications to the classification of linear algebraic groups
and to the study of associated homogeneous varieties.

Tits algebras are examples of cohomological invariants of algebraic groups of degree $2$.
The idea to use cohomological invariants in the classification of algebraic groups goes back to Jean-Pierre Serre.
In particular, Serre conjectured the existence of an invariant of degree $3$ for groups of type $\F_4$ and $\E_8$.
This invariant was later constructed by Markus Rost for all $G$-torsors, where
$G$ is a simple simply-connected algebraic group, and is now called the Rost invariant (see \cite{GMS03}).

Furthermore, the Milnor conjecture (proven by Voevodsky) provides a classification of quadratic forms over fields
in terms of the Galois cohomology, i.e., in terms of cohomological invariants.

In the present article we discuss different approaches to Tits algebras and generalize some of them
to invariants of higher degree. In particular, we consider the Morava $K$-theories, which are the
universal oriented cohomology theories in the sense of Levine--Morel with respect to the Lubin--Tate formal
group law (see Section~\ref{morava}). It turns out that the Morava $K$-theories (more precisely
the Morava-motives) detect, whether certain cohomological
invariants of an algebraic group are zero or not. In particular, the second Morava $K$-theory is responsible
for the Rost invariant (see Theorem~\ref{moravarost}). We remark that in the same spirit Panin showed in \cite{Pa94} that the
Grothendieck's $K^0$ functor detects the triviality of the Tits algebras.

Besides that, we discuss another approach to cohomological invariants which uses an exact sequence
of Voevodsky~\eqref{voev} (see below). For example, this sequence was used in \cite{Sem13} to construct an invariant
of degree $5$ modulo $2$ for some groups of type $\E_8$ and to solve a problem posed by Serre. Besides this,
we discuss oriented cohomology theories of affine varieties and the Rost motives for different
cohomology theories in the sense of Levine--Morel.

We try to keep the exposition elementary and try to avoid technical generalizations, which would rather hide the ideas.
For example, we mostly assume that the groups under considerations are of inner type, though this can be avoided
introducing more notation in the formulae.

It is quite amazing that different ideas from algebra, geometry and topology come together, when dealing with cohomological invariants
of algebraic groups. For example, the Tits algebras are related to representation theory, $K$-theory, but also
to motivic cohomology of simplicial schemes. Looking at the invariants of higher degree one finds relations with algebraic
cobordism of Levine--Morel, classifying spaces of algebraic groups (Totaro, Morel, Voevodsky) and motives;
see e.g. \cite{Me13}, \cite{MNZ13}, \cite{SV14}.

{\bf Acknowledgements.} I would like to thank sincerely Alexey Ananyevskiy, Alexander Neshitov,
and Maksim Zhykhovich for discussions and e-mail conversations on the subject of the article.
I started to work on this article during my visit to University Paris 13 in 2014. I would like to express
my sincerely gratitude to Anne Qu\'eguiner-Mathieu for her hospitality and numerous useful discussions.  

\section{Definitions and notation}

In the present article we assume that $F$ is a field of characteristic $0$. This assumption is not needed
at the most places. In some places it can be removed by changing the \'etale topology by the fpqc topology.
Nevertheless, we would like to avoid a too technical exposition. Aside from that, at some places (e.g. when we consider the
Morava $K$-theory) the assumption on the characteristic is needed. By $\Fsep$ we denote a separable closure of $F$.

Let $G$ be a semisimple linear algebraic group over a field $F$ (see \cite{Springer}, \cite{Inv}, \cite{GMS03}).
A $G$-torsor over $F$ is an algebraic variety $P$ equipped with an action of $G$ such that
$P(\Fsep)\ne\emptyset$ and the action of $G(\Fsep)$ on $P(\Fsep)$ is simply transitive.

The set of isomorphism classes of $G$-torsors over $F$ is a pointed set (with the base point given by
the trivial $G$-torsor $G$) is in natural one-to-one correspondence with the (non-abelian) Galois cohomology set $H_{\et}^1(F,G)$.

Let $A$ be some algebraic structure over $F$ (e.g. an algebra or quadratic space) such that $\Aut(A)$ is an
algebraic group over $F$. Then an algebraic structure $B$ is called a {\it twisted form} of $A$, if over a separable
closure of $F$ the structures $A$ and $B$ are isomorphic. There is a natural bijection between $H_{\et}^1(F,\Aut(A))$ and the set
of isomorphism classes of the twisted forms of $A$.

For example, if $A$ is an octonion algebra over $F$, then
$\Aut(A)$ is a group of type $\G_2$ and $H_{\et}^1(F,\Aut(A))$ is in $1$-to-$1$ correspondence with the twisted forms
of $A$, i.e., with the octonion algebras over $F$ (since any two octonion algebras over $F$ are isomorphic over a
separable closure of $F$ and since any algebra, which is isomorphic to an octonion algebra over a separable closure
of $F$, is an octonion algebra).
                                                                             
By $\qq/\zz(n)$ we denote the Galois-module $\varinjlim\mu_l^{\otimes n}$ taken over all $l$ (see \cite[p.~431]{Inv}).

In the article we use notions from the theory of quadratic forms over fields (e.g. Pfister-forms, Witt-ring). We follow \cite{Inv}, \cite{Lam},
and \cite{EKM}. Further, we use the notion of motives; see \cite{Ma68}, \cite{EKM}.

\section{Algebraic constructions of Tits algebras}
In this section we will briefly describe two classical constructions of Tits algebras following Tits \cite{Ti71}.

\begin{ntt}[Construction using the representation theory of algebraic groups]
Let $G_0$ be a split semisimple algebraic group of rank $n$ over $F$ and $T$ be a split maximal torus of $G_0$. Denote by $\widehat T$
the group of characters of $T$. This is a free abelian group of rank $n$.
Then there is a natural one-to-one correspondence between the isomorphism classes of the irreducible finite dimensional representations of $G_0$
and the elements in $\Lambda_+\cap\widehat T$, where $\Lambda_+$ denotes the cone of dominant weights.
This correspondence associates with an irreducible representation of $G_0$ its highest weight. We remind that
we have a global assumption that $\mathrm{char}\,F=0$. Nevertheless, we remark that this
one-to-one correspondence holds over fields of arbitrary characteristic, but for a given highest weight the
dimension of the respective irreducible representation depends on the characteristic of the field.

Let now $G$ be an arbitrary (not necessarily split) semisimple algebraic group over $F$ which is a twisted form of $G_0$.
A Tits algebra of $G$ corresponding to an element $\omega\in\Lambda_+\cap\widehat T$
is a central simple algebra $A$ over $F$ such that there exists a group homomorphism $\rho\colon G\to\GL_1(A)$
such that the representation $\rho\otimes\Fsep\colon G\otimes\Fsep\to\GL_1(A\otimes_F\Fsep)$ of the split
group $G\otimes\Fsep$ is the representation with the highest weight $\omega$.

Let $\Gamma$ denote the absolute Galois group of $F$. Tits showed that for any $\Gamma$-invariant
$\omega$ the Tits algebra exists and is unique up to an isomorphism. Moreover, there exists a group homomorphism
(called the Tits homomorphism)
$$\beta\colon(\widehat T/\Lambda_r)^\Gamma\to\Br(F)$$
$$\lambda\mapsto [A_\lambda]$$
where $\Lambda_r$ is the root lattice and $A_\lambda$ is the Tits algebra of $G$ of the weight $\chi(\lambda)$,
where $\chi(\lambda)\in\Lambda_+\cap\widehat T$ is a unique representative of $\lambda$ in the coset $\widehat T/\Lambda_r$.
If $G$ is a group of inner type, then the action of $\Gamma$ on $\widehat T/\Lambda_r$ is trivial.

\begin{example}
1). If $G=\SL_1(A)$, where $A$ is a central simple algebra of degree $n+1$, then the Tits algebra
of the first fundamental weight $\omega_1$ (corresponding to the standard representation of $\SL_{n+1}$)
is the algebra $A$ itself. The Tits algebra for the last fundamental weight $\omega_{n}$ (corresponding
to the dual representation) is the opposite algebra $A^{\op}$.

2). If $G=\Spin_{2n+1}(q)$, where $q$ is a regular quadratic form of dimension $2n+1$, then
the Tits algebra of the weight $\omega_n$ (enumeration of simple roots follows Bourbaki) is the even
Clifford algebra of $q$. This corresponds to the spinor representation. If $n>1$, then the Tits algebra of $G$ for the standard representation with the highest
weight $\omega_1$ is the split matrix algebra of degree $2n+1$.

3). If $G$ is a group of inner type $\E_6$ (resp. of type $\E_7$), then the Tits algebra of the weight $\omega_1$
(resp. $\omega_7$) is a central simple algebra of degree $27$, index dividing $27$ and exponent dividing $3$
(resp. of degree $56$, index dividing $8$ and exponent dividing $2$).

4). The Tits algebras of groups of types $\G_2$, $\F_4$ and $\E_8$ are split matrix algebras.
\end{example}

\end{ntt}

\begin{ntt}[Construction using the boundary homomorphism]
Another classical construction of Tits algebras goes as follows.

Let $G$ be a split semisimple algebraic group over $F$ with center $Z=Z(G)$.
The short exact sequence
$$1\to Z\to G\to G/Z\to 1$$ of algebraic groups induces the long exact sequence of cohomology
$$H_{\et}^1(F,G)\to H_{\et}^1(F,G/Z)\xrightarrow{\partial} H_{\et}^2(F,Z).$$
We remark that $G/Z$ is the adjoint group of the same type as $G$.

For an irreducible representation $\rho\colon G\to\GL_n$ denote by $\lambda_\rho$ the restriction
of $\rho$ to the center $Z$. Then $\lambda_\rho$ is a homomorphism from $Z$ to $Z(\GL_n)=\mathbb{G}_m$,
i.e., $\lambda_\rho\in\Hom(Z,\mathbb{G}_m)=\widehat T/\Lambda_r$.

Consider the composite map
$$H_{\et}^1(F,G/Z)\xrightarrow{\partial}H_{\et}^2(F,Z)\xrightarrow{(\lambda_\rho)_*}H_{\et}^2(F,\mathbb{G}_m)=\Br(F).$$

Then an element $\xi\in H_{\et}^1(F,G/Z)$ maps under this composite map to the class of the Tits algebra $A_\rho$ of the (inner)
twisted group ${}_\xi G$ corresponding to the representation $\rho$.
(Instead of a split $G$ one can start with a quasi-split group $G$. Then the analogous construction
will cover all twisted forms of $G$.)

\begin{example}
Let $G=\SL_n$ and let $\rho$ be the standard $n$-dimensional representation of $G$. Then $Z=\mu_n$, $G/Z=\PGL_n$,
and an element $\xi\in H_{\et}^1(F,\PGL_n)$ corresponds to a central simple algebra $A$ over $F$ of degree $n$.
The $H_{\et}^2(F,Z)={}_n\Br(F):=\{x\in\Br(F)\mid nx=0\}\subset\Br(F)$ and the boundary homomorphism $\partial$ maps $\xi$ to the class of $A$ in the Brauer group of $F$.
The composite map $(\lambda_\rho)_*\circ\partial$ maps $\xi$ also to $[A]$, which is the Tits algebra for $\rho$.

Let now $G=\PGL_n$ and $\rho$ be an irreducible representation of $G$. Then $Z=1$ and the Tits algebra of $\rho$ is a
split matrix algebra. This corresponds to the fact that the elements of the root lattice $\Lambda_r$ map to
$0$ under the Tits homomorphism.
\end{example}
\end{ntt}

\section{Geometric constructions of Tits algebras}

\begin{ntt}[Tits algebras and the Picard group]
Another construction of Tits algebras is related to the Hochschild--Serre spectral sequence.
For a smooth variety $X$ over $F$ one has $$H_{\et}^p(\Gamma,H_{\et}^q(X_{\sep},\mathcal{F}))\Rightarrow H_{\et}^{p+q}(X,\mathcal{F})$$
where $X_{\sep}=X\times\Fsep$ and $\mathcal{F}$ is an \'etale sheaf. The induced $5$-term exact sequence is
$$0\to H_{\et}^1(\Gamma,H_{\et}^0(X_{\sep},\mathcal{F}))\to H_{\et}^1(X,\mathcal{F})\to H_{\et}^0(\Gamma,H_{\et}^1(X_{\sep},\mathcal{F}))\to H_{\et}^2(\Gamma,H_{\et}^0(X_{\sep},\mathcal{F}))$$

Let $\mathcal{F}=\mathbb{G}_m$ and $X$ be a smooth projective variety. Then
$$H_{\et}^1(\Gamma,H_{\et}^0(X_{\sep},\mathbb{G}_m))=H_{\et}^1(\Gamma,\Fsep^\times)=0$$ by Hilbert 90,
$H_{\et}^1(X,\mathbb{G}_m)=\Pic(X)$, $H_{\et}^0(\Gamma,H_{\et}^1(X_{\sep},\mathbb{G}_m))=H_{\et}^0(\Gamma,\Pic X_{\sep})=(\Pic X_{\sep})^{\Gamma}$,
and $H_{\et}^2(\Gamma,H_{\et}^0(X_{\sep},\mathbb{G}_m))=H_{\et}^2(\Gamma,\Fsep^\times)=\Br(F)$. Thus, we obtain an exact sequence
\begin{equation}\label{hochsch}
0\to\Pic X\to(\Pic X_{\sep})^\Gamma\to\Br(F)
\end{equation}

The map $\Pic X\to (\Pic X_{\sep})^\Gamma$ is the restriction map and the homomorphism $$(\Pic X_{\sep})^\Gamma\xrightarrow{f}\Br(F)$$
was described by Merkurjev and Tignol in \cite[Section~2]{MT95} when $X$ is the variety of Borel subgroups of a semisimple
algebraic group $G$. Namely, the Picard group of $X_{\sep}$ can be identified with the free abelian group with basis
$\omega_1,\ldots,\omega_n$ consisting of the fundamental weights. If $\omega_i$ is $\Gamma$-invariant (e.g. if $G$
is of inner type), then
$f(\omega_i)=[A_i]$ is the Tits algebra of $G$ corresponding to the (fundamental) representation with the highest weight $\omega_i$
(see \cite{MT95} for a general description of the homomorphism $f$).

Moreover, one can continue the exact sequence~\eqref{hochsch}, namely, the sequence
$$0\to\Pic X\to(\Pic X_{\sep})^\Gamma\to\Br(F)\to\Br(F(X))$$
is exact, where the last map is the restriction homomorphism (see \cite{MT95}).
\end{ntt}

\begin{ntt}[Tits algebras and $K^0$]\label{titsk0}
There is another interpretation of the Tits algebras related to Grothendieck's $K^0$ functor.
Let $G$ be a semisimple algebraic group over $F$ of inner type and $X$ be the variety of Borel subgroups of $G$.
By Panin \cite{Pa94} the $K^0$-motive of $X$ is isomorphic to a direct sum of $|W|$ motives, where $W$ denotes the Weyl
group of $G$. Denote these motives by $L_w$, $w\in W$.

For $w\in W$ consider $$\rho_w=\sum_{\{\alpha_k\in\Pi\mid w^{-1}(\alpha_k)\in\Phi^-\}}w^{-1}(\omega_k)\in\Lambda,$$
where $\Pi$ is the set of simple roots and $\Phi^-$ is the set of negative roots. Using these elements
one can describe the {\it Steinberg basis} of $K^0(X_K)$
over a splitting field $K$ of $G$; see \cite[Section~12.5]{Pa94}, \cite[Section~2]{QSZ12}.

Over a splitting field $K$ of $G$, the motive $(L_w)_K$ is isomorphic to a Tate motive and the restriction
homomorphism $K^0(L_w)\to K^0((L_w)_K)=\zz$ is an injection $\zz\to\zz$ given by the multiplication by $\ind A_w$,
where $[A_w]=\beta(\rho_w)\in\Br(F)$ for the Tits homomorphism $\beta$. In particular, different motives $L_w$ can be
parametrized by the Tits algebras.

Moreover, if all Tits algebras
of $G$ are split, then the $K^0$-motive of $X$ is a direct sum of Tate motives over $F$.
\end{ntt}

\begin{ntt}[Tits algebras and simplicial varieties]
Let $Y$ be a smooth irreducible variety over $F$. Consider the standard simplicial scheme $\mathcal{X}_Y$ associated with $Y$,
i.e. the simplicial scheme 
$$Y\llarrow Y\times Y\lllarrow Y\times Y\times Y\cdots$$

Then for all $n\ge 2$ there is a long exact sequence of cohomology groups (see \cite[Cor.~2.2]{Ro07} and \cite[Proof of Lemma~6.5]{Vo11}):
\begin{equation}\label{voev}
0\to H_{\mathcal{M}}^{n,n-1}(\mathcal{X}_Y,\qq/\zz)\xrightarrow{f} H^n_{\et}(F,\qq/\zz(n-1))\to H_{\et}^n(F(Y),\qq/\zz(n-1)),
\end{equation}
where $H_{\mathcal{M}}^{n,n-1}$ is the motivic cohomology and the homomorphism $f$ is induced by the change of topology
(from Nisnevich to \'etale).

Let $n=2$ and $Y$ be the variety of Borel subgroups of an algebraic group $G$ of inner type. Then $H_{\et}^2(F,\qq/\zz(1))=\Br(F)$
and we have a long exact sequence
$$0\to H_{\mathcal{M}}^{2,1}(\mathcal{X}_Y)\xrightarrow{f}\Br(F)\to\Br(F(Y))$$
Thus, $H_{\mathcal{M}}^{2,1}(\mathcal{X}_Y)=\Lambda/\widehat T$,
where $\Lambda$ is the weight lattice, and $f$ turns out to be the Tits homomorphism.
This gives one more interpretation of the Tits algebras via a change of topology.
\end{ntt}

\section{Higher cohomological invariants}

The Tits algebras are examples of cohomological invariants of $G$-torsors of degree $2$ (since they lie in
$H_{\et}^2(F,\mathbb{G}_m)=\Br(F)$). In general, a cohomological invariant of $G$-torsors of degree $n$ with values
in a Galois-module $M$ is a transformation of functors $H_{\et}^1(-,G)\to H_{\et}^n(-,M)$ from the category of field extensions of
$F$ to the category of pointed sets.

For example, if $G$ is a split orthogonal group of degree $n$ and $M=\zz/2$, then $H_{\et}^1(F,G)$ classifies the isomorphism classes
of regular quadratic forms of dimension $n$ over $F$, $H_{\et}^0(F,M)=\zz/2$, $H_{\et}^1(F,M)=F^\times/F^{\times 2}$ by Hilbert 90,
$H_{\et}^2(F,M)={}_2\Br(F)=\{x\in\Br(F)\mid 2x=0\}$, and the invariants $H_{\et}^1(F,G)\to H_{\et}^n(F,M)$ are given for $n=0,1,2$ resp. by the dimension mod $2$, by the discriminant,
and by the Clifford invariant of the respective quadratic form.

For $n\ge 3$ one can define cohomological invariants for quadratic forms, for which the previous cohomological
invariants (of degree strictly less than $n$) are trivial. Namely, such forms lie in the $n$-th power of the fundamental
ideal $I$ in the Witt ring of $F$, and there are invariants $e_n\colon I^n\to H_{\et}^n(F,\zz/2)$ for all $n$ such that for
$q\in I^n$ the invariant $e_n(q)=0$ iff $q\in I^{n+1}$, $n\ge 0$. Moreover, if $q$ and $q'$ are two quadratic forms,
then one can test whether $q\simeq q'$ looking at $e_n$ of the difference $q-q'$ starting from $n=0$. If $e_n(q-q')\ne 0$,
then $q$ and $q'$ are not isomorphic. Otherwise, $q-q'\in I^{n+1}$ and one proceeds to $e_{n+1}$.
Since by the Arason-Pfister Hauptsatz $\cap_{n\ge 0}I^n=0$ (see \cite[Hauptsatz~5.1, Ch.~X]{Lam}),
the invariants $e_n$ allow us to check an isomorphism between two quadratic forms.

Moreover, under some conditions it is possible to reconstruct the original quadratic form from its invariants.
Namely, start for simplicity with an even dimensional regular quadratic form $q$ over $F$.
Since $H_{\et}^n(F,\zz/2)\simeq I^n/I^{n+1}$ (see \cite{Lam}), we can choose representatives of
elements in $H_{\et}^n(F,\zz/2)$ as linear combinations of $n$-fold Pfister forms.
Computing $e_1(q)\in H_{\et}^1(F,\zz/2)$ we can modify $q$ by $e_1(q)$, which is a linear combination of $1$-fold Pfister forms. Since $e_1(q-e_1(q))=0$,
the form $q-e_1(q)\in I^2$ and we can proceed to the next invariant $e_2$.
If this process stops, i.e. if starting from some point we will get zeroes (e.g. if the base field has finite cohomological dimension), then saving the representatives
of the invariants $e_1(q)$, $e_2(q-e_1(q))$ and so on, will allow us to restore the original form $q$.

The same philosophy can be applied to other algebraic groups. For example, if $G$ is a simple simply connected
algebraic group, then there is an invariant $$H_{\et}^1(-,G)\to H_{\et}^3(-,\qq/\zz(2))$$ of degree $3$, called the {\it Rost invariant}
(see \cite{GMS03}).
If $G$ is the spinor group, this invariant is called the {\it Arason invariant}.

The Rost invariant can be constructed as follows. Let $G$ be a simple simply-connected algebraic group of inner type and $Y$ be
a $G$-torsor. Then there is a long exact sequence (see \cite[Section~9]{GMS03})
\begin{multline}\label{rostseq}
0\to A^1(Y,K_2)\to A^1(Y_{\sep},K_2)^\Gamma\\
\xrightarrow{g}\Ker\big(H^3_{\et}(F,\qq/\zz(2))\to H^3_{\et}(F(Y),\qq/\zz(2))\big)\to\CH^2(Y)
\end{multline}
where the $K$-cohomology group $A^1(-,K_2)$ is defined in \cite[Section~4]{GMS03}, $\Gamma$ is the absolute Galois group,
and $Y_{\sep}=Y\times_F F_{\sep}$.
Moreover, $A^1(Y_{\sep},K_2)^\Gamma=\zz$ and $\CH^2(Y)=0$. The Rost invariant of $Y$ is the image of $1\in A^1(Y_{\sep},K_2)$
under the homomorphism $g$. We remark that sequence~\eqref{rostseq} for the Rost invariant is analogous to the
sequence~\eqref{hochsch} for the Tits algebras arising from the Hochschild--Serre spectral sequence.

We remark also that if $G$ is a group of inner type with trivial Tits algebras (simply-connected or not), then there is a well-defined Rost invariant
of $G$ itself (not of $G$-torsors); see \cite[Section~2]{GP07}.

The idea to use cohomological invariant to study linear algebraic groups and torsors is due to Jean-Pierre Serre.
For example, the Serre--Rost conjecture for groups of type $\F_4$ says that the map
$$H_{\et}^1(F,\F_4)\hookrightarrow H_{\et}^3(F,\zz/2)\oplus H_{\et}^3(F,\zz/3)\oplus H_{\et}^5(F,\zz/2)$$
induced by the invariants $f_3$, $g_3$ and $f_5$ described in \cite[\S40]{Inv} ($f_3$ and $g_3$ are the modulo $2$ and modulo $3$
components of the Rost invariant), is injective. This allows to exchange the study of the {\it set}
$H_{\et}^1(F,\F_4)$ of isomorphism classes of groups of type $\F_4$ over $F$ (equiv. of isomorphism
classes of $\F_4$-torsors or of isomorphism classes of Albert algebras)
by the {\it abelian group} $H_{\et}^3(F,\zz/2)\oplus H_{\et}^3(F,\zz/3)\oplus H_{\et}^5(F,\zz/2)$.

In the same spirit one can formulate the Serre conjecture II, saying in particular that $H_{\et}^1(F,\E_8)=1$ if
the field $F$ has cohomological dimension $2$. Namely, for such fields $H_{\et}^n(F,M)=0$ for all $n\ge 3$ and all torsion modules
$M$. In particular, for groups over $F$ there are no invariants of degree $\ge 3$, and the Serre conjecture II predicts
that the groups of type $\E_8$ over $F$ themselves are split.

Nowadays there exist a number of techniques to construct and study cohomological invariants. In the literature
one can find constructions using $K^0$, Chow rings, motivic cohomology. Moreover, the algebraic cobordism,
general oriented cohomology theories, classifying spaces of algebraic groups and even motives are useful.
For example, in the next section we will describe a relation between invariants and the Morava $K$-theories.

\section{Morava $K$-theory}\label{morava}

In this section we will introduce a geometric cohomology theory --- the Morava $K$-theory, and prove that it detects the triviality of some
cohomological invariants (in particular, of the Rost invariant) of algebraic groups.

Consider the algebraic cobordism $\Omega$ of Levine--Morel (see \cite{LM}).
By \cite[Thm.~1.2.6]{LM} the algebraic cobordism is a universal oriented cohomology theory and there is a
(unique) morphism of theories $\Omega^*\to A^*$ for any oriented cohomology theory $A^*$ in the sense of Levine--Morel.

Each oriented cohomology theory $A$ is equipped with a $1$-dimensional commutative formal group law $\FGL_A$. E.g.,
for the Chow theory $\CH^*$ this is the additive formal group law, for $K^0$ the multiplicative formal group law and for $\Omega$ the universal
formal group law. Moreover, these theories are universal for the respective formal group laws.

For a theory $A^*$ we consider the category of $A^*$-motives with coefficients in a commutative ring $R$,
which is defined in the same way as the category of Grothendieck's Chow motives with $\CH^*$
replaced by $A^*\otimes_\zz R$ (see \cite{Ma68}, \cite{EKM}).
In the present section the ring $R$ is $\zz$, $\qq$, or $\zz_{(p)}$ for a prime number $p$.

For a prime number $p$ and a natural number $n$ we consider the $n$-th Morava $K$-theory $K(n)$ with respect to $p$.
Note that we do not include $p$ in the notation.
We define this theory as the universal oriented cohomology theory for the Lubin--Tate formal group law of height $n$
with the coefficient ring $\zz_{(p)}[v_n,v_n^{-1}]$ (see below for the definition of the Lubin--Tate formal group law).

For a variety $X$ over $F$ one has $$K(n)(X)=\Omega(X)\otimes_\laz\zz_{(p)}[v_n,v_n^{-1}],$$
and $v_n$ is a $\nu_n$-element in the Lazard ring $\mathbb{L}$, i.e., an element whose {\it Milnor number}
is not divisible by $p^2$ (see \cite[Section~2]{Sem13} and \cite[Section~4.4.4]{LM}). The degree of $v_n$
is negative and equals $-(p^n-1)$.
In particular, $K(n)(\Spec F)=\zz_{(p)}[v_n,v_n^{-1}]$. We remark that usually one considers the Morava
$K$-theory with the coefficient ring $\ff_p[v_n,v_n^{-1}]$. For any prime $p$ we define $K(0)$ as $\CH\otimes\qq$.

If $n=1$ and $p=2$, one has $K(1)(X)=K^0(X)[v_1,v_1^{-1}]\otimes\zz_{(2)}$, since the Lubin--Tate formal
group law is isomorphic to the multiplicative formal group law in this case.

We describe now the formal group law for the $n$-th Morava $K$-theory modulo $p$ (the Lubin--Tate formal
group law) following \cite{Haz} and \cite{Rav}.
The logarithm of the formal group law of the {\it Brown--Peterson cohomology} equals
$$l(t)=\sum_{i\ge 0}m_it^{p^i},$$
where $m_0=1$ and the remaining variables $m_i$ are related to $v_j$ as follows:
$$m_j=\frac 1p\cdot\big(v_j+\sum_{i=1}^{j-1}m_iv_{j-i}^{p^i}\big).$$
Let $e(t)$ be the compositional inverse of $l(t)$.
The Brown--Peterson formal group law is given by $e(l(x)+l(y))$.

The $n$-th Morava formal group law is obtained from the $BP$ formal group law by sending all $v_j$ with $j\ne n$
to zero. Modulo the ideal $J$ generated by $p, x^{p^n}, y^{p^n}$ the formal group law for the $n$-th Morava
$K$-theory equals
$$\FGL_{K(n)}(x,y)=x+y-v_n\sum_{i=1}^{p-1}\frac 1p\binom pi x^{ip^{n-1}}y^{(p-i)p^{n-1}}\mod J.$$

\begin{dfn}
For an oriented cohomology theory $A$ and a motive $M$ in the category of $A$-motives over $F$ we say that
$M$ is {\it split}, if it is a finite direct sum
of (twisted) Tate motives over $F$. Note that this property depends on the theory $A$. E.g., there exist
smooth projective varieties whose motives are split for some oriented cohomology theories, but not for all
oriented cohomology theories.
\end{dfn}

\begin{ntt}[Euler characteristic]\label{eulerc}
The Euler characteristic of a smooth projective irreducible variety $X$ with respect to an oriented cohomology
theory $A^*$ is defined as the push-forward
$$\pi_*^A(1_X)\in A^*(\Spec F)$$ of the structural morphism $\pi\colon X\to\Spec F$.
E.g., for $A=K^0[v_1,v_1^{-1}]$
the Euler characteristic of $X$ equals $$v_1^{\dim X}\cdot\sum(-1)^i\dim H^i(X,\mathcal{O}_X)$$
see \cite[Ch.~15]{Ful}. If $X$ is geometrically irreducible and geometrically cellular, then this element equals $v_1^{\dim X}$
(see \cite[Example~3.6]{Za10}).

For the Morava $K$-theory $K(n)$ and a smooth projective irreducible variety $X$ of dimension $d=p^n-1$
the Euler characteristic modulo $p$ equals the element $v_n\cdot u\cdot s_d$ for some $u\in\zz_{(p)}^\times$,
where $s_d$ is the Milnor number of $X$ (see \cite[Sec.~4.4.4]{LM}). If $\dim X$ is not divisible by $p^n-1$, then the Euler characteristic of $X$ equals zero modulo $p$
(see \cite[Prop.~4.4.22]{LM}).
\end{ntt}

\begin{ntt}[Rost nilpotence for oriented cohomology theories]\label{rostnil}
Let $A$ be an oriented cohomology theory and consider the category of $A$-motives over $F$.
Let $M$ be an $A$-motive over $F$. We say that the Rost nilpotence principle holds for $M$,
if the kernel of the restriction homomorphism $$\End(M)\to\End(M_E)$$ consists of nilpotent correspondences for all
field extensions $E/F$.

Usually Rost nilpotency is formulated for Chow motives. By \cite[Sec.~8]{CGM05} it holds for
all twisted flag varieties. Note that the proof of \cite{CGM05} works for $A$-motives of twisted flag varieties
for all oriented cohomology theories $A$ in the sense of Levine--Morel satisfying the localization property.

Rost nilpotency is a tool which allows to descent motivic decompositions over $E$ to motivic decompositions over
the base field $F$. E.g., assume that Rost nilpotency holds for $M$ and that we are given a decomposition
$M_E\simeq\oplus{M_i}$ over $E$ into a finite direct sum. The motives $M$ and $M_i$ are defined as a pair $(X,\rho)$ and $(X_E,\rho_i)$,
where $X$ is a smooth projective variety over $F$, $\rho\in A(X\times X)$ and $\rho_i\in A(X_E\times X_E)$ are some projectors. If we assume further that all $\rho_i$ are defined over $F$,
then $M\simeq\oplus N_i$ for some motives $N_i$ over $F$, and the scalar extension $(N_i)_E$ is isomorphic to $M_i$
for every $i$ (see \cite[Section~8]{CGM05}, \cite[Section~2]{PSZ08}).
\end{ntt}

Let $R_m$ denote the (generalized) {\it Rost motive} of a non-zero pure symbol $\alpha\in H_{\et}^m(F,\mu_p^{\otimes m})$ in the category of Chow motives
with $\zz_{(p)}$-coefficients. By definition $R_m$ is indecomposable and for all field extensions $K/F$ the following
conditions are equivalent:
\begin{enumerate}
\item $(R_m)_K$ is decomposable;
\item $(R_m)_K\simeq\bigoplus_{i=0}^{p-1}\zz_{(p)}(b\cdot i)$ with $b=\frac{p^{m-1}-1}{p-1}$;
\item $\alpha_K=0\in H_{\et}^m(K,\mu_p^{\otimes m})$.
\end{enumerate}
The fields $K$ from this definition are called splitting fields of $R_m$.

The Rost motives were constructed by Rost and Voevodsky (see \cite{Ro07}, \cite{Vo11}). Namely, for all pure symbols $\alpha$
there exists a smooth projective $\nu_{m-1}$-variety $X$ (depending on $\alpha$) over $F$ such that the Chow motive
of $X$ has a direct summand isomorphic to $R_m$ and for every field extension $K/F$ the motive $(R_m)_K$
is decomposable iff $X_K$ has a $0$-cycle of degree coprime to $p$. Besides that, it follows from
Brosnan's extension \cite[Thm.~3.1]{Br03} of a result \cite[Prop.~1]{Ro98} of Rost
that Rost nilpotency holds for $R_m$. The variety $X$ is called a {\it norm variety} of $\alpha$.

E.g., if $p=2$ and $\alpha=(a_1)\cup\ldots\cup(a_m)$ with $a_i\in F^\times$, then one can take for $X$
the projective quadric given by the equation $\langle\!\langle a_1,\ldots,a_{m-1}\rangle\!\rangle\perp\langle -a_m\rangle=0$,
where $\langle\!\langle a_1,\ldots,a_{m-1}\rangle\!\rangle$ denotes the Pfister form.
(We use the standard notation from the quadratic form theory as in \cite{Inv} and \cite{EKM}.)

By \cite[Sec.~2]{ViYa07} there is a unique lift of the
Rost motive $R_m$ to the category of $\Omega$-motives and, since $\Omega$ is the universal oriented cohomology theory,
there is a well-defined Rost motive in the category of $A^*$-motives for any oriented cohomology theory $A^*$.
We will denote this $A$-motive by the same letter $R_m$. By $\ta(l)$, $l\ge 0$, we denote the Tate motives in the
category of $A$-motives. If $A=\CH\otimes\zz_{(p)}$, we keep the usual notation $\ta(l)=\zz_{(p)}(l)$.

\begin{prop}\label{pro62}
Let $p$ be a prime number, $n$ and $m$ be natural numbers and $b=\frac{p^{m-1}-1}{p-1}$.
For a non-zero pure symbol $\alpha\in H_{\et}^m(F,\mu_p^{\otimes m})$ consider the respective Rost motive $R_m$. Then
\begin{enumerate}
\item If $n<m-1$, then the $K(n)$-motive $R_m$ is a sum of $p$ Tate motives
$\oplus_{i=0}^{p-1}\ta(b\cdot i)$.
\item If $n=m-1$, then the $K(n)$-motive $R_m$ is a sum of the Tate motive $\ta$
and an indecomposable motive $L$ such that
$$K(n)(L)\simeq(\zz^{\oplus(p-1)}\oplus(\zz/p)^{\oplus{(m-2)(p-1)}})\otimes\zz_{(p)}[v_n,v_n^{-1}].$$
For a field extension $K/F$ the motive $L_K$ is isomorphic to a direct sum of twisted Tate motives
iff it is decomposable and iff the symbol $\alpha_K=0$.
\item If $n>m-1$, then the $K(n)$-motive $R_m$ is indecomposable and its realization is isomorphic to the group
$\CH(R_m)\otimes\zz_{(p)}[v_n,v_n^{-1}]$.
For a field extension $K/F$ the motive $(R_m)_K$ is decomposable iff $\alpha_K=0$.
In this case $(R_m)_K$ is a sum of $p$ Tate motives.
\end{enumerate}
\end{prop}
\begin{proof}
Denote by $\overline R_m$ the scalar extension of $R_m$ to its splitting field.
By \cite[Prop.~11.11]{Ya12} (cf. \cite[Thm.~3.5, Prop.~4.4]{ViYa07}) the restriction map for the $BP$-theory
\begin{equation}\label{eq1}
\res\colon BP(R_m)\to BP(\overline R_m)=BP(\Spec F)^{\oplus p}
\end{equation}
is injective, and the image equals 
\begin{equation}\label{eq2}
BP(R_m)\simeq BP(\Spec F)\oplus I(p,m-2)^{\oplus (p-1)},
\end{equation}
where $I(p,m-2)$ is the ideal in the ring $BP(\Spec F)=\zz_{(p)}[v_1,v_2,\ldots]$ generated by the elements
$\{p,v_1,\ldots, v_{m-2}\}$.

(1) Assume first that $n<m-1$.
Since the ideal $I(p,m-2)$ contains $v_n$ for $n<m-1$ and $v_n$ is invertible in $K(n)(\Spec F)$, we immediately get that all
elements in $K(n)(\overline R_m)$ are rational, i.e., are defined over the base field. By the properties of the
Rost motives
\begin{equation}\label{formom}
\Omega^l(R_m\times R_m)=\bigoplus_{i+j=l}\Omega^i(R_m)\otimes\Omega^j(\overline R_m)
\end{equation}
for all $l$. Since $\Omega$ is a universal theory, the same formula holds for $BP$ and for $K(n)$.
Therefore all elements in $K(n)(\overline R_m\times \overline R_m)$ are rational, and hence by Rost nilpotency
for $R_m$ this gives the first statement of the proposition.

(2) Assume now that $n=m-1$.
Let $X$ be a norm variety for the symbol $\alpha$. In particular, $\dim X=p^{m-1}-1=p^n-1$.
Since the Morava--Euler characteristic of $X$ equals $u\cdot v_n$ for some
$u\in\zz_{(p)}^\times$ (see Section~\ref{eulerc}), the element $v_n^{-1}\cdot u^{-1}(1\times 1)\in K(n)(X\times X)$
is a projector defining the Tate motive $\ta$. Thus, we get the decomposition $R_m\simeq\ta\oplus L$ for some motive $L$.
We claim that $L$ is indecomposable.

Indeed, by \cite[Thm.~4.4.7]{LM} the elements of
$K(n)^{p^n-1}(R_m\times R_m)$ are linear combinations of elements of the form $v_n^s\cdot [Y\to X\times X]$,
where $Y$ is a resolution of singularities of a closed subvariety of $X\times X$, and $-s(p^n-1)+\codim Y=p^n-1$.
In particular, $s=0,1,-1$ and $\codim Y=0,\,p^n-1,\,2(p^n-1)$.

By formula~\eqref{formom} and by the injectivity of the restriction map for $BP$,
it follows that there are at most three rational projectors in $K(n)(\overline R_m\times\overline R_m)$
(cf. \cite[Section~6 and Proof of Lemma~6.12]{Nes14}).
These are the diagonal, the projector $v_n^{-1}\cdot u^{-1}(1\times 1)$ constructed above and their difference
(which defines the motive $L$).
Therefore by Rost nilpotency the motive $L$ is indecomposable over $F$.

Taking the tensor product $-\otimes_{BP(\Spec F)} K(n)(\Spec F)$ with formula~\eqref{eq1} and
using \eqref{eq2} one immediatelly gets the formula for $K(n)(L)$.

(3) The same arguments show that $R_m$ is indecomposable for the Morava $K$-theory
$K(n)$ for $n>m-1$. Notice that $\CH=BP/D\cdot BP$, where $D$ is the ideal in $BP(\Spec F)$ generated
by $v_1,v_2,\ldots$. Since $\res$ in formula~\eqref{eq1} is injective, we get that
$$\CH(R_m)=\mathrm{Im}(\res)/D\cdot\mathrm{Im}(\res).$$ Similary we obtain a formula for the Morava $K$-theory $K(n)$,
which is obtained from $BP$ by sending all $v_i$ with $i\ne n$ to $0$ and localizing in $v_n$. This formula
is the same as for $\CH\otimes\zz_{(p)}[v_n,v_n^{-1}]$ (cf. \cite[Example~6.14]{Nes14}).
\end{proof}

\begin{rem}
This proposition demonstrates a difference between $K^0$ and the Morava $K(n)$-theory, when $n>1$. By \cite{Pa94} $K^0$
of all twisted flag varieties is $\zz$-torsion-free. This is not the case for $K(n)$, $n>1$.

Moreover, the same arguments as in the proof of the proposition
show that the connective $K$-theory $CK(1)$ (see \cite{Cai08}) of Rost motives $R_m$ for $m>2$ contains non-trivial $\zz$-torsion.
\end{rem}

\begin{rem}
The Chow groups of the Rost motives are known; see \cite[Thm.~8.1]{KM02}, \cite[Thm.~RM.10]{KM13}, \cite[Cor.~10.8]{Ya12},
\cite[Section~4.1]{Vi07}.
\end{rem}

The proof of the following proposition is close to \cite[Section~8]{A12}.

\begin{prop}\label{moravaaffine2}
Let $A$ be an oriented
cohomology theory in the sense of Levine--Morel satisfying the localization property.
Let $Z$ be a smooth variety over a field $F$.
Assume that there exists a smooth projective variety $Y$ with invertible Euler characteristic
with respect to $A$ and such that for every point $y\in Y$ (not necessarily closed)
the natural pullback $$A(F(y))\to A(Z_{F(y)})$$ is an isomorphism

Then the pullback of the structural morphism $Z\xrightarrow{\pi}\Spec F$
induces an isomorphism $$A(Z)=A(F).$$
\end{prop}
\begin{proof}
We have the following localization diagram
\[
 \xymatrix{
\lim\limits_{\substack{\longrightarrow\\ Y'\subset Y}}A(Y') \ar@{->}[r]\ar@{->}[d] & A(Y) \ar@{->}[r]\ar@{->}[d] & A(F(Y)) \ar@{->}[r] \ar@{->}[d]&0\\
\lim\limits_{\substack{\longrightarrow\\ Y'\subset Y}}A(Z\times Y') \ar@{->}[r] & A(Z\times Y) \ar@{->}[r] & A(Z_{F(Y)}) \ar@{->}[r] &0
 }
 \]
where the vertical arrows are pullbacks of the respective projections and the limits are taken over all closed
subvarieties of $Y$ of codimension $\ge 1$.

By the choice of $Y$
the right vertical arrow is an isomorphism. By induction on dimension of $Y$ the left vertical arrow is surjective.
It follows by a diagram chase that the middle vertical arrow is surjective as well.

Let $a\colon Y\to\Spec F$ be the structural morphism, $b\colon Z\times Y\to Y$ and $c\colon Z\times Y\to Z$
be the projections. Consider now another commutative diagram:
$$
\xymatrix{
A(F) \ar@{->}[dd]^{\simeq}\ar@{->}[rd]^-{a^*}\ar@{->}[rrr]^-{\pi^*}& & & A(Z) \ar@{->}[ld]^-{c^*} \ar@{->}[dd]^{\simeq}\\
& A(Y) \ar@{->}[r]^-{b^*}\ar@{->}[ld]^-{a_*} & A(Z\times Y)\ar@{->}[rd]^-{c_*} &\\
A(F)\ar@{->}[rrr]^-{\pi^*} & & & A(Z)
}
$$
By the above considerations the homomorphism $b^*$ is surjective. The left and the right vertical arrows
are isomorphisms, since they are multiplications by the $A$-Euler characteristic of $Y$ which is invertible.

Therefore by a diagram chase the bottom horizontal arrow is surjective. But $A(F)$ is a direct summand of
$A(Z)$. Therefore the bottom arrow is an isomorphism.
\end{proof}

Let now $(a_1)\cup\ldots\cup(a_m)\in H_{\et}^m(F,\zz/2)$ be a pure symbol, $a_i\in F^\times$.
The quadratic form $q=\langle\!\langle a_1,\ldots,a_{m-1}\rangle\!\rangle\perp\langle -a_m\rangle$
is called a {\it norm form} and the respective projective quadric given by $q=0$ is called a (projective)
{\it norm quadric}. The respective {\it affine norm quadric} is an open subvariety of the projective norm quadric
given by the equation $$\langle\!\langle a_1,\ldots,a_{m-1}\rangle\!\rangle=a_m,$$ i.e. setting the last coordinate to $1$.

\begin{cor}\label{moravaaffine}
Let $0\le n<m-1$ and set $p=2$. Consider the affine norm quadric $X^{\aff}$ of dimension $2^{m-1}-1$ corresponding
to a pure symbol in $H_{\et}^m(F,\zz/2)$. Then the pullback of the structural morphism $X^{\aff}\xrightarrow{\pi}\Spec F$
induces an isomorphism $$K(n)(X^{\aff})=K(n)(F).$$
\end{cor}
\begin{proof}
Let $\alpha:=(a_1)\cup\ldots\cup(a_m)\in H_{\et}^m(F,\zz/2)$ be our pure symbol, $a_i\in F^\times$,
$q$ the norm form for $\alpha$, and $Q$ the respective projective norm quadric given by $q=0$.
Let $Y$ be the projective norm quadric of dimension $2^n-1$ corresponding to the subsymbol
$$(a_1)\cup\ldots\cup(a_{n+1})\in H_{\et}^{n+1}(F,\zz/2).$$

We need to check the conditions of Proposition~\ref{moravaaffine2}.
By the choice of $Y$ it is a $\nu_n$-variety (see e.g. \cite[Section~2]{Sem13}). Therefore by \cite[Prop.~4.4.22]{LM}
its Morava--Euler characteristic is invertible.

Moreover, the quadratic form $q$ is split completely over $F(y)$ for any point $y$ of $Y$.
In particular, $X^{\aff}_{F(y)}$ is a split odd-dimensional affine quadric.
The complement $Q':=Q\setminus X^{\aff}$
is a projective Pfister quadric of dimension $2^{m-1}-2$, and both $Q$ and $Q'$ are
split over $F(y)$. Therefore the Chow motives of $Q$ and $Q'$ over $F(y)$ are direct sums of twisted
Tate motives. Moreover, $\CH(X^{\aff}_{F(y)})=\zz$ (see \cite[Theorem~A.4]{Ka01} for a more general result).

Denote by $\Mot(-)$ the Chow motive of a variety $-$.
The localization sequence for Chow groups implies that the induced map
of $\Mot(Q')\to\Mot(Q)$ over $F(y)$ is surjective in positive codimensions. By \cite[Section~2]{ViYa07}
the same holds for $\Omega$-motives of $Q'$ and $Q$ over $F(y)$. Therefore $\Omega(X^{\aff}_{F(y)})=\laz$,
and, thus, $K(n)(X^{\aff}_{F(y)})=K(n)(F(y))$. We are done.
\end{proof}

Let now $B$ be a central simple $F$-algebra of a prime degree $p$ and $c\in F^\times$.
Consider the Merkurjev--Suslin variety
$$\MS(B,c)=\{\alpha\in B\mid \Nrd(\alpha)=c\},$$
where $\Nrd$ stands for the reduced norm on $B$.

\begin{cor}
In the above notation the structural morphism induces an isomorphism $A(\MS(B,c))\simeq A(F)$, when $A$ is Grothendieck's $K^0$ or
the first Morava $K$-theory with respect to the prime $p$.
\end{cor}
\begin{proof}
Let $Y=\SB(B)$ denote the Severi--Brauer variety of $B$. We need to check the conditions of Proposition~\ref{moravaaffine2}.
The varietry $Y$ is a geometrically cellular $\nu_1$-variety. Thus, its $A$-Euler characteristic is invertible.

Over a point $y\in Y$ the variety $\MS(B,c)$ is isomorphic to $\SL_p$, since $\MS(B,c)$ over $F(y)$
is an $\SL_p$-torsor over $F(y)$ and $H^1_{\et}(F(y),\SL_p)=1$. Since $\GL_p$ is an open subvariety in
$\mathbb{A}^{p^2}$, by the localization sequence $\Omega(\GL_p)=\laz$. Moreover, $\GL_p$ is isomorphic
as a variety (not as a group scheme) to $\SL_p\times\mathbb{G}_m$ with the isomorphism sending a matrix $\alpha$
to the pair $(\beta,\det\alpha)$ where $\beta$ is obtained from $\alpha$ by dividing its first row by $\det\alpha$.
The composite morphism $$\SL_p\hookrightarrow\GL_p\xrightarrow{\simeq}\SL_p\times \mathbb{G}_m\to\SL_p,$$ where the first morphism
is the natural embedding and the last morphism is the projection, is the identity. Taking pullbacks in this
sequence, one gets that $\Omega(\SL_p)=\laz$ and, hence,
$A(\SL_p)=A(F(y))$ for $A$ as in the statement of the present corollary. We are done.
\end{proof}

Let $J$ be an Albert algebra over $F$ (see \cite[Chapter~IX]{Inv}) and $N_J$ denote the cubic norm form on $J$.
For $d\in F^\times$ consider the variety $$Z=\{\alpha\in J\mid N_J(\alpha)=d\}.$$ The group $G$ of isometries
of $N_J$ is a group of type $^1{}\E_6$ and it acts on $Z$ geometrically transitively.

The proof of the following statement belongs to A.~Ananyevskiy.
\begin{cor}
In the above notation the natural map $K^0(Z)\to K^0(F)$ is an isomorphism.
\end{cor}
\begin{proof}
Let $Y$ be the variety of Borel subgroups of the group $G$. We need to check the conditions of
Proposition~\ref{moravaaffine2}. The $K^0$-Euler characteristic of $Y$ is invertible, since $Y$ is
geometrically cellular.

Let $y\in Y$ be a point. Then $G$ splits over $F(y)$, the variety $Z$ has a rational point over $F(y)$
and its stabilizer is the split group of type $\F_4$, i.e., $Z$ is isomorphic to $\E_6/\F_4$ over $F(y)$,
where $\E_6$ and $\F_4$ stand for the split groups of the respective Dynkin types.

By the branching rules the restriction homomorphism $R(\E_6)\to R(\F_4)$ is surjective, where $R(-)$
denotes the representation ring (see \cite[Section~2]{A12}). Therefore Merkurjev's spectral
sequence (see \cite[Section~6]{A12}) implies that $$K^0(\E_6/\F_4)\simeq K^0(F(y))\otimes_{R(\E_6)}R(\F_4)\simeq K^0(F(y)).$$
We are done.
\end{proof}

Consider the Witt-ring of the field $F$ and denote by $I$ its fundamental ideal.

\begin{prop}\label{moravaproj}
Let $m\ge 1$ and set $p=2$.
A regular even-dimensional quadratic form $q$ belongs to $I^m$ iff the Morava motives
$K(n)$ of the respective projective quadric are split for all $0\le n<m-1$.
\end{prop}
\begin{proof}
Assume that $q$ does not belong to $I^m$. Let $1\le s<m$ be the maximal number with $q\in I^s$.
By \cite[Thm.~2.10]{OVV07} there exists a field extension $K$ of $F$ such that $q_K$ as an element of the
Witt-ring of $F$ is an anisotropic $s$-fold Pfister
form. By Prop.~\ref{pro62} its $(s-1)$-st Morava motive is not split. Contradiction.

Conversely, assume that $q$ belongs to $I^m$ and let $Q$ be the respective projective quadric.
Then we can present $q$ as a finite sum of (up to proportionality)
$s$-fold Pfister forms with $s\ge m$. We prove our statement using induction on the length of such a presentation in the Witt-ring.               
If $q$ is an $s$-fold Pfister-form, then, since $s\ge m>n+1$, by Prop.~\ref{pro62} the $K(n)$-motive of $Q$ is split.

Let $\alpha$ be an $s$-fold Pfister form in the decomposition of $q$. Let $X^{\aff}$ be the affine norm quadric of
dimension $2^{n}-1$ corresponding to a subsymbol of $\alpha$ from $H_{\et}^{n+1}(F,\zz/2)$ (note that $n+1<m\le s$).
Then the length of $q$ over $F(X^{\aff})$ is strictly smaller than the length of $q$ over $F$.

Consider the following commutative diagram of localization sequences:
\[
 \xymatrix{
\lim\limits_{\substack{\longrightarrow\\ Y'\subset Q\times Q}}K(n)(Y') \ar@{->}[r]\ar@{->}[d] & K(n)(Q\times Q) \ar@{->}[r]\ar@{->}[d] & K(n)(F(Q\times Q)) \ar@{->}[r] \ar@{->}[d]&0\\
\lim\limits_{\substack{\longrightarrow\\ Y'\subset Q\times Q}}K(n)(X^{\aff}\times Y') \ar@{->}[r] & K(n)(X^{\aff}\times Q\times Q) \ar@{->}[r] & K(n)(X^{\aff}_{F(Q\times Q)}) \ar@{->}[r] &0
 }
 \]
where the vertical arrows are pullbacks of the respective projections and the limits are taken
over all closed subvarieties of $Q\times Q$ of codimension $\ge 1$.

The right vertical arrow is an isomorphism by Corollary~\ref{moravaaffine}. The left vertical arrow is surjective by
induction on the dimension of the variety $Q\times Q$ (we do not use here that $Q$ is a quadric). Therefore by a diagram
chase the middle vertical arrow is surjective.

But by the localization sequence $$K(n)(X^{\aff}\times Q\times Q)\to K(n)((Q\times Q)_{F(X^{\aff})})$$
is surjective. By the induction hypothesis on the length of $q$, the restriction homomorphism
$$K(n)((Q\times Q)_{F(X^{\aff})})\to K(n)((Q\times Q)_K)$$
to a splitting field $K$ of $Q_{F(X^{\aff})}$ is surjective. Therefore the restriction homomorphism $$K(n)(Q\times Q)\to K(n)((Q\times Q)_K)$$
to the splitting field $K$ is surjective. In particular, since the projectors
for the Morava-motive of $Q$ lie in $K(n)(Q\times Q)$, it follows from Rost nilpotency that the $K(n)$-motive of $Q$ over $F$ is split.
\end{proof}

\begin{rem}
The same statement with a similar proof holds for the variety of totally isotropic
subspaces of dimension $k$ for all $1\le k\le (\dim q)/2$.
\end{rem}

\begin{thm}\label{moravarost}
Let $p$ be a prime number.
Let $G$ be a simple algebraic group over $F$ and let $X$ be the variety of Borel subgroups of $G$. Then
\begin{enumerate}
\item $G$ is of inner type iff the $K(0)$-motive of $X$ is split.
\item Assume that $G$ is of inner type. All Tits algebras of $G$ are split iff the $K^0$-motive with integral coefficients
of $X$ is split.
\item Assume that $G$ is of inner type and the $p$-components of the Tits algebras 
of $G$ are split. Then the $p$-component of the Rost invariant of $G$ is zero iff the $K(2)$-motive of $X$ is split.
\item Let $p=2$. Assume that $G$ is of type $\E_8$ with trivial Rost invariant. Then $G$ is split by an odd degree field
extension iff the $K(m)$-motive of $X$ is split for some $m\ge 4$ iff the $K(m)$-motive of $X$ is split for all $m\ge 4$.
\end{enumerate}
\end{thm}
\begin{proof}
(1) Since $K(0)=\CH\otimes\qq$ by definition, the statement follows from the fact that $G$ is of inner type
iff the absolute Galois group of $F$ acts trivially on $\CH(X_{\Fsep})\otimes\qq$ and from the fact that over a
splitting field $\Fsep$ of $G$ the variety $X_{\Fsep}$ is cellular.

(2) Follows from \cite{Pa94}; see also Section~\ref{titsk0}.

(3) First we make several standard reductions.
Since all prime numbers coprime to $p$ are invertible in the coefficient ring of the Morava $K$-theory,
by transfer argument we are free to take finite field extensions of the base field of degree coprime to $p$.
Hence we can assume that not only the $p$-components of the Tits algebras are split, but that the
Tits algebras are completely split (and the same for the Rost invariant).

If $G$ is a group of inner type $\A$ or $\C$ with trivial Tits algebras, then $G$ is split and the statement
follows. If $G$ is a group of type $\B$ or $\D$, then the statement follows from Proposition~\ref{moravaproj}
(In Prop.~\ref{moravaproj} we assume that the quadratic form $q$ is even-dimensional.
We use it only to conclude that $q\in I$ as a starting point in the proof.)

Let now $G$ be a group of an exceptional type.
Assume that the $K(2)$-motive of $X$ is split, but the Rost invariant of $G$ is not trivial.
By \cite[Thm.~5.7]{PS10} if $G$ is not split already, there is a field extension $K$ of $F$ such that the Rost invariant of $G_K$
is a pure non-zero symbol (For example, if $p=2$ and $G$ is of type $\E_8$, then one
can take $K=F(Y)$ with $Y$ the variety of maximal parabolic subgroups of $G$ of type $6$;
enumeration of simple roots follows Bourbaki). Then the motive of $X$ is a direct sum of Rost motives corresponding to this symbol
of degree $3$ (see \cite{PSZ08}). This gives a contradiction with Prop.~\ref{pro62}.

Conversely, if the Rost invariant of $G$
is zero and $G$ is not of type $\E_8$, then by \cite[Thm.~0.5]{Ga01} the group $G$ is split and the statement follows.
If $G$ is of type $\E_8$ with trivial Rost invariant, then by \cite[Theorem.~8.7]{Sem13} $G$ has an invariant $u\in H^5_{\et}(F,\zz/2)$
such that for a field extension $K/F$ the invariant $u_K=0$ iff $G_K$ splits over a field extension
of $K$ of odd degree. Exactly as in the proof of Prop.~\ref{moravaproj} we can reduce to the case when $u$ is a pure symbol.
But then by \cite{PSZ08} the motive of the variety $X$ (Chow motive and cobordism motive and hence Morava-motive) is a direct sum of Rost motives for $u$ and
by Prop.~\ref{pro62} the $K(2)$-Rost motive for a symbol of degree $5>3$ is split.

(4) If $G$ is split by an odd degree field extension, then the $K(m)$-motives of $X$ are split for all $m$, since $p=2$.
Conversely, if $G$ does not split over an odd degree field extension of $F$, then the invariant $u$ is not zero.
By \cite[Thm.~2.10]{OVV07} there is field extension $K$ of $F$ such that $u_K$ is a non-zero pure symbol.
Over $K$ the motive of $X$ is a direct sum of Rost motives corresponding to $u_K$. By Prop.~\ref{pro62} the $K(m)$-Rost
motives for a symbol of degree $5$ are not split, if $m\ge 4$.
\end{proof}

Finally we remark that sequence~\eqref{voev} can be used to define the Rost invariant in general,
the invariant $f_5$ for groups of type $\F_4$ and an invariant of degree $5$ for groups of type $\E_8$ with trivial Rost invariant (see \cite{Sem13}).
Namely, for the Rost invariant let $G$ be a simple simply-connected algebraic group over $F$.
Let $Y$ be a $G$-torsor and set $n=3$. Then sequence~\eqref{voev} gives an exact sequence
$$0\to H_{\mathcal{M}}^{3,2}(\mathcal{X}_Y,\qq/\zz)\to \Ker\big(H^3_{\et}(F,\qq/\zz(2))\to H_{\et}^3(F(Y),\qq/\zz(2))\big)\to 0$$
But by sequence~\eqref{rostseq} $\Ker\big(H^3_{\et}(F,\qq/\zz(2))\to H_{\et}^3(F(Y),\qq/\zz(2))\big)$ is a finite cyclic
group. Therefore $H_{\mathcal{M}}^{3,2}(\mathcal{X}_Y,\qq/\zz)$ is a finite cyclic group and the Rost invariant
of $Y$ is the image of $1\in H_{\mathcal{M}}^{3,2}(\mathcal{X}_Y,\qq/\zz)$ in $H_{\et}^3(F,\qq/\zz(2))$.

To construct invariants of degree $5$ for $\F_4$ (resp. for $\E_8$) one takes $n=5$ and $Y$ to be the variety of parabolic subgroups
of type $4$ for $\F_4$ (the enumeration of simple roots follows Bourbaki) and resp. the variety of parabolic subgroups of any type for $\E_8$.
In both cases $H^{5,4}_{\mathcal{M}}(\mathcal{X}_Y,\qq/\zz)$ is cyclic of order $2$ and the invariant is the
image of $1\in H^{5,4}_{\mathcal{M}}(\mathcal{X}_Y,\qq/\zz)$ in $H^5_{\et}(F,\qq/\zz(4))$; see \cite{Sem13}.

\bibliographystyle{chicago}

\medskip

\medskip

\noindent
\sc{Nikita Semenov\\
Institut f\"ur Mathematik, Johannes Gutenberg-Universit\"at
Mainz, Staudingerweg 9, D-55128, Mainz, Germany}\\
{\tt semenov@uni-mainz.de}

\end{document}